\documentclass[reqno,a4paper,twoside]{amsart}
\usepackage{mathtools}
\usepackage{amsmath}
\usepackage{amsthm}
\usepackage{amsfonts}
\usepackage{xfrac}
\usepackage{faktor}
\usepackage[utf8]{inputenc}

\usepackage{amsrefs}

\newtheorem{theorem}{Theorem}[section]
\newtheorem{lemma}[theorem]{Lemma}
\newtheorem{proposition}[theorem]{Proposition}
\newtheorem{proposition*}{Proposition}

\theoremstyle{definition}

\theoremstyle{remark}

\begin{document}

\title[A construction of the free skew-brace]{A construction of the free skew-brace}

\author{Juan Orza}
\address{Departamento de Matem\'atica\\ Facultad de Ciencias Exactas y Naturales-UBA, Pabell\'on~1-Ciudad Universitaria\\ Intendente Guiraldes 2160 (C1428EGA) Buenos Aires, Argentina.}
\email{}

\begin{abstract}
The aim of this paper is to give a simple construction of the free skew-brace over a set $X$.
\end{abstract}

\maketitle

\section*{Introduction}

The concept of skew-brace was introduced by Guarnieri and Vendramin \cite{GV} in order to study non-involutive set theoretical solutions of de braid equation using method of non-commutative ring theory. For the original definition we refer to this paper. In \cite{V} the author ask for a construction of the skew-brace over a set $X$. The free skew-brace over $X$ exist by general result of universal algebra, but the constructions of free objects valid in the context of universal algebra are very involved. So, in this note we present a simple construction of the free skew-brace using the following result established in \cite{R}:

\medskip

\noindent \textbf{Proposition.}\enspace A group $(A,\circ)$, endowed with a pair of left actions $b\mapsto a\cdot b$ y $b\mapsto a:b$ on $A$ (as a set), is a skew-brace if and only if the following equality holds:
\begin{equation}
	(a\cdot b)\circ a = (b:a)\circ b \quad \text{for all $a, b \in A$.}\label{eq}
\end{equation}

\section{A construction of the free skew-brace over $X$}

To begin with, we recursively define a set $Y$ in the following way: first we set
\[
Y_1 \coloneqq X\sqcup X^*,
\]
where $X^*\coloneqq \{x^*:x\in X\}$ is a disjoint copy of the elements of $X$. Next we define the sets
\begin{align*}
	X_n^{\cdot } &\coloneqq \{a\cdot b \ | \ a,b\in Y_{n-1}, b\neq a^* \cdot  c \mbox{ for any } c\in Y_{n-1}  \}\\
	X_n^{:} &\coloneqq \{a:b \ | \ a,b\in Y_{n-1}, b\neq a^* : c \mbox{ for any } c\in Y_{n-1}  \}\\	
	Y_n &\coloneqq Y_{n-1} \sqcup (X_n^{\cdot } \sqcup X_n^{:}) \sqcup (X_n^{\cdot } \sqcup X_n^{:})^*
\end{align*}
where $a\cdot b$ and $a:b$ are formal elements. Then we set $Y \coloneqq \bigcup_{n\geq 1} Y_n$ and we let $G$ denotes the free group $F(Y)$ over $Y$ (with underlying set the collection of reduced words, with the element $x^*$ playing the roll of the inverse of $x$). In the sequel, we are going to denote the multiplication of $G$ by juxtaposition $a\circ b = ab$.

The next step is to define maps $\cdot$ and $:$ from $Y\times Y$ to $Y$ by
\[
a\cdot b = \left\{
\begin{array}{lr}
	a\cdot b & \text{ if }b\neq a^*\cdot c,\\
	c & \text{ if }b=a^*\cdot c,
\end{array}
\right.
\]
\[
a:b = \left\{
\begin{array}{lr}
a:b & \text{ if }b\neq a^*:c,\\
c & \text{ if }b=a^*:c.
\end{array}
\right.
\]

We want to extend both $\cdot$ and $:$ to left actions of $G$ on $G$ in such away that $(a\cdot b)a = (b:a)b$ for all $a,b\in G$, which happen if and only if
\begin{align}
	a\cdot b &= (b:a)ba^{-1}\quad\text{for all $a,b\in G$,} \label{a}
\shortintertext{or, equivalently,}
	b:a &= (a\cdot b)ab^{-1}\quad\text{for all $a,b\in G$.} \label{b}
\end{align}
Taking $b=1$ in this equalities  we see that necessarily $a\cdot 1 = a:1 = 1$ for all $a\in G$. Let $a,b,c \in G$. From (\ref{b}) it follows that
\[
\begin{array}{lrl}
	 & (ab) : c &= a:(b:c) \\
\Leftrightarrow & (c\cdot (ab))c (ab)^{-1} &= a : ((c\cdot b)cb^{-1})\\
\Leftrightarrow & (c\cdot (ab))cb^{-1}a^{-1} &= (((c\cdot b)cb^{-1})\cdot a)(c\cdot b)cb^{-1}a^{-1}\\
\Leftrightarrow & c\cdot (ab) &= ((b:c)\cdot a)(c\cdot b).
\end{array}
\]
Similarly, from (\ref{a}) it follows that:
\[
(ab) \cdot c = a\cdot (b\cdot c) \Leftrightarrow c:(ab) = ((b\cdot c): a)(c\cdot b).
\]
This motivates to extend both $\cdot$ and $:$ to maps from $Y\times G$ to $G$, by defining
\begin{align}
	x \cdot  (gy) &\coloneqq ((y:x)\cdot g)(x\cdot y), \label{def.}\\
	x : (gy) &\coloneqq ((y\cdot x):g)(x:y), \label{def:}
\end{align}
for each $x,y\in X$ and each reduced word $gy$, and then to extend these definitions to left actions of $G$ on $G$, by defining
\begin{align*}
	(x_1x_2\dots x_n) \cdot  g &= x_1 \cdot  (x_2\cdot (\dots (x_n\cdot g)\dots)),\\
	(x_1x_2\dots x_n) : g &= x_1 : (x_2:(\dots (x_n:g)\dots)),\\
	1 \cdot  g &= 1:g = g,
\end{align*}
for all $x_i \in Y$ and $g\in G$.

Now we check that
$$
x\cdot (ab) = ((b:x)\cdot a)(x\cdot b)\quad\text{for all $x\in Y$ and $a,b\in G$.}
$$
We proceed by induction on the length of $b$. This is clear for $b=1$. Assume that it is true for $b\in G$, and let us see that it is true for $by$ with $y\in Y$:
	\begin{align*}
		x\cdot (a(by)) &= ((y:x)\cdot (ab))(x\cdot y)\\
		&= ((b:(y:x))\cdot a)((y:x)\cdot b)(x\cdot y)\\
		&= (((by):x)\cdot a)(x\cdot (by))
	\end{align*}
In a similar way it is proved that $x:(ab)=((b\cdot x):a)(x:b)$.

\begin{lemma}\label{compatibilidad de cdot y : con el producto} For all $a,b,c\in G$, the identities
\begin{equation}
a\cdot (bc) = ((c:a)\cdot b)(a\cdot c)\quad\text{and}\quad a : (bc) = ((c\cdot a): b)(a : c)\label{ref: compatibilidad de cdot y : con el producto}
\end{equation}
are fulfilled.
\end{lemma}

\begin{proof} We prove the first equality and left the second one, which is similar, to the reader. Clearly the first equality in~\eqref{ref: compatibilidad de cdot y : con el producto} is true if $a=1$ or $c=1$. Take $c=x \in Y$ and suppose that the first equality in~\eqref{ref: compatibilidad de cdot y : con el producto} holds for $a\in G$. Then,
	\begin{align*}
		(ya)\cdot (bx) &= y\cdot (a\cdot (bx))\\
		&=y\cdot (((x:a)\cdot b)(a\cdot x))\\
		&=(((a\cdot x):y)\cdot ((x:a)\cdot b))(y\cdot (a\cdot x))\\
		&=(((a\cdot x):y)(x:a)\cdot b)((ya)\cdot x)\\
		&=((x:(ya))\cdot b)((ya)\cdot x)
	\end{align*}
Let $a,b,c\in G$, $z\in X$, and suppose that the first equality in~\eqref{ref: compatibilidad de cdot y : con el producto} holds $c\in G$. We have
	\begin{align*}
		a\cdot (b(cz)) &= ((z:a)\cdot (bc))(a\cdot z)\\
		&=((c:(z:a))\cdot b)((z:a)\cdot c)(a\cdot z)\\
		&=((cz : a)\cdot b)(a\cdot (cz)).
	\end{align*}
which proves that the first equality in~\eqref{ref: compatibilidad de cdot y : con el producto} also holds for $cz$.
\end{proof}

It is not true that $(a\cdot b)a=(b:a)b$ for all $a,b\in G$. So, in order to finish our task we need to take a quotient $G/\sim$ of $G$ by a suitable congruence relation.

\begin{proposition}\label{(a cdot b)a=(b:a)b sobre generadores}
Let $A$ be a group endowed with a pair of left actions $b\mapsto a\cdot b$ y $b\mapsto a:b$ on $A$ (as a set) such that
\begin{equation*}
a\cdot (bc) = ((c:a)\cdot b)(a\cdot c)\quad\text{and}\quad a : (bc) = ((c\cdot a): b)(a : c)\quad\text{for all $a,b,\in A$,}
\end{equation*}
and let $X$ be a set of generators of $A$ as a monoid. If $(x\cdot y)x=(y:x)y$ for all $x,y\in X$, then $(a\cdot b)a=(b:a)b$ for all $a,b\in A$.
\end{proposition}

\begin{proof}
Suppose that $(x\cdot b)x = (b:x)b$ for some $b\in G$ and all $x,y\in X$. Then, for all $x,y\in X$,
\begin{align*}
	(x\cdot (by))x &= ((y:x)\cdot b)(x\cdot y)x\\
	&= (b:(y:x))by\\
	&=((by):x)by.
\end{align*}
So, $(x\cdot b)x=(b:x)b$ for all $x\in X$ and $b\in G$. Let $a\in G$. Suppose that $(a\cdot b)a=(b:a)b$ for all $b\in A$. Then
\begin{align*}
	((ax)\cdot b)ax &= (a\cdot (x\cdot b))ax \\
	&= ((x\cdot b):a)(x\cdot b)x\\
	&= ((x\cdot b):a)(b:x)b\\
	&= (b:(ax))b,
\end{align*}
which finishes the proof.
\end{proof}

\noindent Let $\sim$ be the minimal equivalence relation in $A$ such that:
\begin{enumerate}

\item $(x\cdot y)x \sim (y:x)y$ for all $x,y\in X$, and

\item If $a\sim b$ and $c\sim d$, then
\begin{align}
	\qquad\quad ac &\sim bd,\\
	a\cdot c &\sim b\cdot d,\\
	a: c &\sim b: d.
\end{align}

\end{enumerate}

\noindent By Proposition~\ref{(a cdot b)a=(b:a)b sobre generadores} and the proposition in the introduction, $G/\sim$ is a skew brace.

\begin{theorem}
The group $G/\sim$, endowed with the binary operations $\cdot$ and $:$ induced by the corresponding operations in $G$, is the free skew brace over $X$.
\end{theorem}

\begin{proof}
Let $A$ be a skew brace and let $f\colon X\to A$ be a map. Next, we extend $f$ to a group morphism $\widetilde{f}\colon G\to A$ such that
$$
\widetilde{f}(x\cdot y) = f(x)\cdot f(y)\quad\text{and}\quad \widetilde{f}(x : y) = f(x): f(y)\quad\text{for all $x,y\in Y$.}
$$
For this, first we extend $f$ to $Y$ as follows:

\begin{itemize}

\item[-] We extend $f$ to $Y_1$ by defining $\widetilde{f}(x^*)\coloneqq f(x)^{-1}$ for each $x\in X$;

\item[-] Assuming that we have defined $\widetilde{f}$ on $Y_{n-1}$, we extend $\widetilde{f}$ to $Y_n$ by defining $\widetilde{f}(x\cdot y) \coloneqq f(x)\cdot f(y)$ and $\widetilde{f}(x : y)\coloneqq f(x): f(y)$, for all $x,y\in Y_{n-1}$.

\end{itemize}
Then, we extend $f$ to $G$ by defining
$$
\widetilde{f}(x_1x_2\dots x_n) \coloneqq  \widetilde{f}(x_1)\widetilde{f}(x_2)\cdots\widetilde{f}(x_n)\quad\text{for $x_i\in Y$.}
$$
Since equality~\eqref{eq} is satisfied in~$A$, the map $\widetilde{f}$ induce a skew brace morphism $\varphi\colon G/\sim\to A$, unique such that $f=\varphi\circ i$, where $i\colon X\to G/\sim$ is the canonical map.
\end{proof}

\end{document}